\documentclass[12pt,reqno]{amsart}

\usepackage{amssymb,amsmath,amsthm,mathtools,wasysym,calc,verbatim,enumitem,tikz,hyperref,url,mathrsfs,cite,fullpage,bbm,comment} 
\mathtoolsset{showonlyrefs}
\usepackage{scalerel}
\usepackage{stackengine}
\stackMath

\newcommand\pig[1]{\scalerel*[5pt]{\normalsize#1}{%
  \ensurestackMath{\addstackgap[0.5pt]{\normalsize#1}}}}
  \newcommand\pigl[1]{\hspace{0.05cm}\mathopen{\pig{#1}}\hspace{0.03cm}}

\newtheorem{theorem}{Theorem}
\newtheorem{definition}[theorem]{Definition}

\newtheorem{lemma}[theorem]{Lemma}
\newtheorem{claim}[theorem]{Claim}

\newtheorem{observation}[theorem]{Observation}

\newtheorem*{chernoff}{Chernoff's inequality}

\theoremstyle{remark}
\newtheorem*{remark*}{Remark}
\newtheorem{remark}[theorem]{Remark}

\numberwithin{theorem}{section}

\renewcommand{\phi}{\varphi}

\newcommand{\gap}{\hspace{0.02cm}}
\newcommand{\gaap}{\hspace{0.03cm}}
\newcommand{\gaaap}{\hspace{0.04cm}}

\newcommand{\ina}{\hspace{0.01cm}\in\hspace{0.01cm}}

\newcommand{\setminusa}{\hspace{0.01cm}\setminus\hspace{0.01cm}}

\newcommand{\eqa}{\hspace{0.02cm}=\hspace{0.015cm}}
\newcommand{\nea}{\hspace{0.025cm}\ne\hspace{0.025cm}}

\newcommand{\eps}{\varepsilon}
\newcommand{\cF}{\mathcal F}

\newcommand{\cA}{\mathcal A}
\newcommand{\cB}{\mathcal B}
\newcommand{\cD}{\mathcal{D}}

\newcommand{\cN}{\mathcal N}

\newcommand{\cQ}{\mathcal Q}
\newcommand{\cR}{\mathcal R}

\newcommand{\cI}{\mathcal I}

\def\1{\mathbbm{1}}

\renewcommand{\le}{\leqslant}
\renewcommand{\ge}{\geqslant}

\newcommand{\Ex}{\mathbb E}
\renewcommand{\Pr}{\mathbb{P}}

\newcommand{\N}{\mathbb N}

\newcommand{\Bin}{\operatorname{Bin}}
\newcommand{\Ber}{\operatorname{Ber}}

\begin{document}

\title{On the multicolour Ramsey numbers $R(3,3,k)$} 

\author{Bruno Andrades, Marcelo Campos and Robert Morris}

\address{IMPA, Estrada Dona Castorina 110, Jardim Bot\^anico, Rio de Janeiro, 22460-320, Brasil}\email{bruno.andrades|marcelo.campos|rob@impa.br}

\thanks{BA was supported by CNPq; MC was partially supported by Serrapilheira (grant R-2412-51283); and RM~was partially supported by CNPq (Proc.~407970/2023-1 and Proc.~305818/2026-0), by FAPERJ (Proc.~E-26/204.288/2024), and by FAPESB (Proc.~012/2022-APP0044/2023)}

\begin{abstract}
In this paper we determine the Ramsey number $R(3,3,k)$ up to a constant factor, showing that
$$R(3,3,k) = \Theta\bigg( \frac{k^3}{(\log k)^2} \bigg).$$
The proof of the lower bound combines the Hefty--Horn--King--Pfender construction for $R(3,k)$ with the method of Alon and R\"odl. Using the same proof, we also determine the $r$-colour Ramsey numbers $R(3,\ldots,3,k)$ up to a constant factor for every fixed $r \ge 3$. 
\end{abstract}

\maketitle

\vspace{-2em}

\section{Introduction}

The off-diagonal Ramsey number $R(3,k)$ is the smallest $n \in \N$ such that every red-blue colouring of the edges of the complete graph $K_n$ contains either a red triangle or a blue copy of $K_k$. The study of these numbers has its origins in the seminal work of Ramsey~\cite{R30} and Erd\H{o}s and Szekeres~\cite{ESz} in the 1930s, and over the past several decades it has inspired the development of a number of important techniques in probabilistic combinatorics, such as the method of alterations~\cite{E59,E61} and the semi-random method~\cite{AKSz,Kim}. We refer the reader to the surveys~\cite{ICM26} and~\cite{SpR3k} for the history of the problem and its impact on the field. 

In this paper we will study the natural multicolour analogue of these Ramsey numbers. The first case of this problem is $R(3,3,k)$, which is the smallest $n \in \N$ such that every red-blue-green colouring of the edges of $K_n$ contains either a red triangle, a green triangle or a blue copy of $K_k$. While the growth rate of $R(3,k)$ was determined over 30 years ago in famous papers of Ajtai, Koml\'os and Szemer\'{e}di~\cite{AKSz} and Kim~\cite{Kim}, which together imply that 
$$R(3,k) = \Theta\bigg( \frac{k^2}{\log k} \bigg),$$
progress on the multicolour problem has been much slower, and for many years it was an open problem even to show that $R(3,3,k) \gg R(3,k)$. This was finally resolved in 2005 in a groundbreaking paper of Alon and R\"odl~\cite{AR}, who obtained the bounds  
\begin{equation}\label{eq:AR:bounds}
\frac{c k^3}{(\log k)^4} \le R(3,3,k) \le \frac{C k^3}{(\log k)^2}
\end{equation}
for some constants $C > c > 0$. To be more precise, the upper bound in~\eqref{eq:AR:bounds} follows from the method of Ajtai, Koml\'os and Szemer\'{e}di~\cite{AKSz}, while to prove the lower bound they showed that there exists a triangle-free graph with few independent $k$-sets, and then took random copies of this graph as the red and green edges (see Section~\ref{sec:AR:method}).

In this paper we will improve the lower bound of Alon and R\"odl by a factor of $(\log k)^2$, and hence determine $R(3,3,k)$ up to a constant factor. 

\begin{theorem}\label{thm:R33k}
$$R(3,3,k) = \Theta\bigg( \frac{k^3}{(\log k)^2} \bigg)$$
\end{theorem}

Our proof of this theorem combines the method of Alon and R\"odl~\cite{AR} with a construction that was introduced recently by Hefty, Horn, King and Pfender~\cite{HHKP}, who used it to improve the lower bound for $R(3,k)$ to within a factor of $2 + o(1)$ of the best-known upper bound, which was proved by Shearer~\cite{Sh83} in 1983. This construction involves taking two random blow-ups of $G(m,p)$, placing them randomly on top of one another, and then deleting any edge whose endpoints have a common neighbour in either of the two copies (see Section~\ref{sec:construction}). While this construction is reminiscent of that of Alon and R\"odl, it was actually inspired by a more recent approach of Campos, Jenssen, Michelen and Sahasrabudhe~\cite{CJMS25}. Let us also remark that, in addition to our use of it in this paper, the Hefty--Horn--King--Pfender method has already found several other significant and surprising applications~\cite{Br26,CJMPS,KSSW,MSV}.

Our method also applies to the $r$-colour Ramsey number $R_r(3,\ldots,3,k)$, which is the smallest $n \in \N$ such that every $r$-colouring of the edges of $K_n$ contains either a monochromatic triangle in one of the first $r - 1$ colours, or a monochromatic copy of $K_k$ in colour $r$. The methods of Ajtai, Koml\'os and Szemer\'{e}di~\cite{AKSz} and Alon and R\"odl~\cite{AR} imply the bounds\footnote{Alon and R\"odl (see~\cite[Theorem~3.2]{AR}) stated slightly weaker bounds in their paper, but with a little care one can obtain~\eqref{eq:R333k:AR} using their method, see~\cite[Theorem~3]{HW}.}
\begin{equation}\label{eq:R333k:AR}
\frac{ck^r}{(\log k)^{2r-2}} \le R_r(3,\ldots,3,k) \le \frac{Ck^r}{(\log k)^{r-1}}
\end{equation}
for some constants $C = C(r)$ and $c = c(r)$; for the reader's convenience we recall the proof of the upper bound in Section~\ref{sec:AR:method}. Our proof of Theorem~\ref{thm:R33k} can also be extended to the general setting, and allows us to determine $R_r(3,\ldots,3,k)$ up to a constant factor (depending on $r$). 

\begin{theorem}\label{thm:R333k}
For every fixed $r \ge 2$ we have
$$R_r(3,\ldots,3,k) = \Theta\bigg( \frac{k^r}{(\log k)^{r-1}} \bigg)$$
\end{theorem}

In order to use the method of Alon and R\"odl to prove Theorem~\ref{thm:R333k}, we need to construct a triangle-free graph with the claimed number of vertices and sufficiently few independent $k$-sets. The main contribution of this paper is to show that the construction of Hefty, Horn, King and Pfender~\cite{HHKP}, after adjusting the parameters correctly, has the required property. To be slightly more precise, the main theorem of~\cite{HHKP} states that there exists a triangle-free graph with $n$ vertices and no independent sets of size $k \ge (1 + \eps) \displaystyle{\sqrt{n \log n}}$. We will prove the following theorem, which shows that the number of smaller independent sets is (roughly speaking) dominated by subsets of neighbourhoods. 

\begin{theorem}\label{thm:acmgraph} 
For every $n \ge k^2 / \log k$, there exists a triangle-free graph $G$ with $n$ vertices and at most
$$2^{O(k)} {pn \choose k}$$
independent sets of size $k$, where $pk = \log k$. 
\end{theorem}

The main step in the proof of Theorem~\ref{thm:acmgraph} is a structural lemma (Lemma~\ref{lem:structural:lemma}), which says that if our graph $G$ satisfies a certain pseudorandom property, then every set of $k$ vertices of $G$ satisfies one of five properties. For three of these properties we will be able to bound the number of such sets with some simple counting; for the other two, we will need a bound for the probability that a set with many `open edges' is independent. This bound, which is stated in Lemma~\ref{lem:manyopen:probbound} and proved in Sections~\ref{sec:inside:edges} and~\ref{sec:outside:edges}, is similar to others that have been used in previous applications of the method (see, for example,~\cite[Lemma~4.1]{HHKP} or~\cite[Lemma~3.4]{MSV}), but in order to deal with small values of $k$ we need a slightly more general version, which may be useful for other applications. The pseudorandom property that we will use to prove the structural lemma is stated in Lemma~\ref{lem:bipartite:lemma}, and proved in Section~\ref{sec:open:pairs}. 

The rest of the paper is organised as follows. First, in Section~\ref{sec:AR:method}, we will recall from~\cite{AR} the proof of the upper bound in Theorem~\ref{thm:R333k}, and show how to deduce the lower bound from Theorem~\ref{thm:acmgraph}. Next, in Sections~\ref{sec:construction} and~\ref{outline:sec}, we will define the random graph $G$ that we use to prove Theorem~\ref{thm:acmgraph}, outline how we will bound the number of independent $k$-sets in $G$, and state Lemmas~\ref{lem:manyopen:probbound} and~\ref{lem:bipartite:lemma}. In Sections~\ref{sec:inside:edges}--\ref{sec:open:pairs} we will prove these two probabilistic lemmas, and in Section~\ref{sec:structural} we will state and prove our key structural lemma. Finally, in Section~\ref{final:proof:sec}, we will complete the proof of Theorem~\ref{thm:acmgraph} via some simple counting.

\section{Deducing Theorem~\ref{thm:R333k} from Theorem~\ref{thm:acmgraph}}\label{sec:AR:method}

In this section we will assume that Theorem~\ref{thm:acmgraph} holds, and deduce the following version of Theorem~\ref{thm:R333k} in which $r$ is allowed to grow with $k$. 

\begin{theorem}\label{thm:R333k:general}
There exists a constant $C > 0$ such that
\begin{equation}\label{eq:R333k:general}
\frac{k^r}{(C\log k)^{r-1}} \le R_r(3,\ldots,3,k) \le \frac{(Crk)^r}{(\log k)^{r-1}}
\end{equation}
for every $r \ge 2$ and $k \ge C^3$.
\end{theorem}

Alon and R\"odl~\cite{AR} proved the upper bound in~\eqref{eq:R333k:general} with an extra factor of $(\log \log k)^{r-2}$, and explained how this factor could be removed using an observation due to Benny Sudakov. A nice exposition of the proof was given by He and Wigderson~\cite{HW}, who also generalised it to arbitrary off-diagonal multicolour Ramsey numbers, though without keeping track of the dependence on $r$. We will show below that their argument gives the upper bound in~\eqref{eq:R333k:general}. 

\pagebreak

The upper bound follows from the following lemma of Ajtai, Koml\'os and Szemer\'{e}di~\cite{AKSz}, which follows easily from the corresponding statement for triangle-free graphs.

\begin{lemma}\label{lem:AKSz:fewtriangles}
There exists a constant $c > 0$ such that if $1 \le \lambda \le d \le n$, then
\begin{equation}\label{eq:AKSz:bound}
\alpha(G) \ge \frac{cn}{d} \log \lambda
\end{equation}
for every graph $G$ with $n$ vertices, $\Delta(G) \le d$ and at most $d^2 n / \lambda$ triangles.
\end{lemma}

We can now prove the upper bound in~\eqref{eq:R333k:general} by induction on $r$. 

\begin{proof}[Proof of the upper bound in Theorem~\ref{thm:R333k:general}]
When $r = 2$, the bound follows from Lemma~\ref{lem:AKSz:fewtriangles}, since if $G$ is a triangle-free graph then $\alpha(G) \ge \Delta(G)$. So let $r \ge 3$, assume that the claimed bound holds for $r - 1$, and let $\chi$ be an $r$-colouring of $E(K_n)$ with no monochromatic triangle in colour $i$ for each $i \in \{1,\ldots,r-1\}$, and no monochromatic copy of $K_k$ in colour $r$. 

Let $G$ be the graph formed by the edges of colours $1,\ldots,r-1$, and observe that 
\begin{equation}\label{eq:DeltaG:bound}
\Delta(G) \le r \cdot R_{r-1}(3,\ldots,3,k) \le r \cdot \frac{(Crk)^{r-1}}{(\log k)^{r-2}} = d,
\end{equation}
by the induction hypothesis. Moreover, for each pair of vertices $u,v \in V(G)$ and each pair of colours $i,j \in [r-1]$ with $i \ne j$, we have
$$|N_i(u) \cap N_j(v)| \le R_{r-2}(3,\ldots,3,k) \le \frac{(Crk)^{r-2}}{(\log k)^{r-3}} = \frac{d \log k}{Cr^2k},$$
where $N_i(u)$ denotes the set of neighbours of $u$ in colour $i$, and $R_1(k) = k$. It follows that the number of triangles in $G$ is at most
$$\sum_{uv \ina E(G)} \sum_{i \eqa 1}^{r-1} \sum_{j \nea i} |N_i(u) \cap N_j(v)| \le r^2 e(G) \cdot \frac{d \log k}{Cr^2k} \le d^2 n \cdot \frac{\log k}{Ck},$$
since every triangle in $G$ uses at least two colours. By Lemma~\ref{lem:AKSz:fewtriangles}, it follows that 
$$n \le \frac{Cd \cdot \alpha(G)}{\log k} \le \frac{(Crk)^r}{(\log k)^{r-1}}$$
since $\alpha(G) < k$. 
\end{proof}

The deduction of the lower bound from Theorem~\ref{thm:acmgraph} is also straightforward, using the approach of Alon and R\"odl~\cite{AR}. 

\begin{proof}[Proof that Theorem~\ref{thm:acmgraph} implies Theorem~\ref{thm:R333k:general}]
The upper bound in~\eqref{eq:R333k:general} was proved above, and the case $r = 2$ of the lower bound was proved by Kim~\cite{Kim}, so it remains to prove the lower bound for all $r \ge 3$. 
To do so, let $n = k^r / (C\log k)^{r-1}$, and observe that 
$n \ge k^2 / \log k$, since $r \ge 3$ and $k \ge C^3$. We may therefore apply Theorem~\ref{thm:acmgraph} to obtain a graph $G$ with $n$~vertices and
$$|\cI_k(G)| \le 2^{O(k)} {pn \choose k},$$
where $\cI_k(G)$ is the set of independent sets of $G$ of size $k$, and $pk = \log k$. 

To define an $r$-colouring of the edges of $K_n$, let $G_1,\ldots,G_{r-1}$ be independent random copies of $G$ with vertex set $V(K_n)$. That is, each graph $G_i$ is formed by randomly permuting the vertices of $G$. Let $G_r = K_n$, and colour the edges of $K_n$ by setting
$$\chi(e) = \min\big\{ i \in [r] : e \in E(G_i) \big\}$$
for each edge $e \in E(K_n)$. Note that there are no monochromatic triangles in any colour $i \in \{ 1,\dots,r-1\}$ since $G$ is triangle-free, and that a copy of $K_k$ in colour $r$ corresponds to a set of $k$ vertices that is independent in each of $G_1,\ldots,G_{r-1}$. Let $X$ denote the number of such sets, and observe that 
$$\Ex\pig[ X \pig] = {n \choose k} \bigg( \frac{|\cI_k(G)|}{\binom nk} \bigg)^{r-1} \le \gap 2^{O(rk)} {pn \choose k}^{r - 1} {n \choose k}^{-(r - 2)},$$
by Theorem~\ref{thm:acmgraph}. Recalling that $pk = \log k$ and $n \le k^r$, it follows that
$$\Ex\pig[ X \pig] \le 2^{O(rk)} \bigg( \frac{p^{r-1}n}{k} \bigg)^k < \bigg( \frac{n \pigl( \hspace{-0.03cm} C \log k \hspace{0.015cm} \pig)^{r-1}}{k^r} \bigg)^k = 1$$
if $C > 0$ is sufficiently large. It follows that $X = 0$ with positive probability, and hence there exists a choice of permutations such that there is no copy of $K_k$ in colour $r$, as required.
\end{proof}

\begin{remark}
Theorem~\ref{thm:acmgraph} can also be used to prove Kim's lower bound on $R(3,k)$. To see this, simply observe that if $G$ has $2^{O(k)}$ independent $k$-sets, then $\alpha(G) = O(k)$, since every subset of an independent set is also independent. 
\end{remark}

\section{The construction}\label{sec:construction}

In this section we will define the random graph $G$ that we will use to prove Theorem~\ref{thm:acmgraph}. As discussed in the introduction, it is exactly the construction of Hefty, Horn, King and Pfender~\cite{HHKP}, except with a different choice of the parameters. It is formed by taking the union of two random blow-ups of the graph $G(m,p)$, and then deleting edges that have a common neighbour in one (or both) of these two graphs. 

\subsection{The construction}

Let $n,k \in \N$ with $n \ge k^2 / \log k$, 
and set\footnote{The extra constant factor in the definition of $p$ will be absorbed into the $O(k)$ term in Theorem~\ref{thm:acmgraph}.} 
\begin{equation}\label{def:parameters}
p = \frac{C \log k}{k} \qquad \text{and} \qquad m = \frac{1}{8p^2},
\end{equation}
where $C$ is a sufficiently large constant. Let $A$ and $B$ be independent copies of the random graph $G(m,p)$. Next, let $A^*$ and $B^*$ be blow-ups of $A$ and $B$, respectively, obtained by replacing each vertex by an independent set of size $n/m$, and each edge by the corresponding complete bipartite graph, and let $G^*$ be the random graph with $n$ vertices and edge set
$$E(G^*) = E(A^*) \cup E(B^*),$$
where the vertices of $A^*$ and $B^*$ are each mapped randomly (and independently) onto $V(G^*)$. Finally, let $G$ be the graph that is obtained from $G^*$ by removing every edge whose end-vertices have a common neighbour in either $A^*$ or $B^*$. 

To be slightly more formal, let us fix two partitions of $V(G)$ into sets of size $n/m$, which we will refer to as \emph{fibres}: 
\begin{equation}\label{def:partitions}
V(G) = F_A^{(1)} \cup \cdots \cup F_A^{(m)} = F_B^{(1)} \cup \cdots \cup F_B^{(m)}
\end{equation}
and let $F_A^{(i)}$ and $F_B^{(i)}$ be the blow-up classes in $A^*$ and $B^*$, respectively. We will also find it convenient to define
\begin{equation}\label{def:fibres}
\cF = \cF_A \cup \cF_B, \quad \text{where} \quad \cF_X = \big\{ F_X^{(1)}, \ldots, F_X^{(m)} \big\} \quad \text{for each } X \in \{A,B\}.
\end{equation}
We will only require these partitions to satisfy a single property, related to the number of sets that intersect few fibres of each. In the next section we define this property, and show that we can choose the partitions so that it is satisfied (see Lemma~\ref{lem:partitions:property}). 

Now, if the random graphs $A$ and $B$ each have vertex set $[m]$, then the edge set of $A^*$ is
$$E(A^*) = \bigcup_{ij \ina E(A)} \big\{ uv : u \in F_A^{(i)}, v \in F_A^{(j)} \big\},$$
and similarly for $B^*$. Next, define the set of \emph{open pairs} with respect to $A$ to be those with no common neighbour in $A^*$,
$$O_A = \bigg\{ uv \in {V(G) \choose 2} : N_A(u) \cap N_A(v) = \emptyset \bigg\},$$
and define $O_B$ similarly. Here, and throughout the paper, we write $N_A(v)$ to denote the neighbourhood of $v$ in the graph $A^*$ when $v \in V(G)$, and the neighbourhood of $v$ in the graph $A$ when $v \in V(A)$ (and similarly for $B$), and trust that this will not cause confusion. We can now define the edge set of $G$ to be
\begin{equation}\label{def:G}
E(G) = E(G^*) \cap O_A \cap O_B.
\end{equation}
Let us record here the simple fact that $G$ is (deterministically) triangle-free.

\begin{lemma} 
$G$ is triangle-free.
\end{lemma}

\begin{proof}
Any triangle in $G^*$ either has (at least) two edges in $A^*$, or two edges in $B^*$, and therefore the third edge is not in $O_A \cap O_B$, and so is not an edge of $G$.  
\end{proof}

Before continuing to the proof, let us briefly discuss the choice of parameters above. First, we cannot take $p$ much smaller than we do, since if $pk \ll \log k$ then every triangle-free graph with $n$ vertices and density $p$ has $n^{-o(k)} {n \choose k}$ independent $k$-sets,\footnote{We were unable to find a reference for this statement, but it can be read out of several known proofs of Ajtai, Koml\'os and Szemer\'{e}di's bound~\eqref{eq:AKSz:bound} for the independence number of a triangle-free graph.} which is larger than the bound in Theorem~\ref{thm:acmgraph} when $p \le n^{-c}$ for some constant $c > 0$. 

Our choice of $m$ is perhaps more surprising to those who are familiar with the Hefty--Horn--King--Pfender method, since in previous applications (for example, in~\cite{CJMPS,HHKP,KSSW,MSV}) the random graph was taken to be significantly sparser. Here our choice is forced by the proof; we require $m \le p^{-2}$ to control the number of open edges, and $m = \Omega(p^{-2})$ in order to count $k$-sets that are contained in $O(1/p)$ fibres, see Lemma~\ref{lem:counting:e}.

\section{An outline of the proof}\label{outline:sec}

To prove Theorem~\ref{thm:acmgraph}, we will show that, with high probability, the random graph $G$ has at most $2^{O(k)} {pn \choose k}$ independent $k$-sets. The proof of this bound has three main steps: a technical structural lemma, which divides the independent $k$-sets of $G$ into five classes, and two probabilistic lemmas, one of which controls the number of open edges in certain sets of vertices, and one which bounds the probability that a set with many open edges is independent. In this section we will describe these three steps, and how they fit together.  

\subsection{The structural lemma, and sets with many open edges}

In order to describe the five classes given by the structural lemma (see Lemma~\ref{lem:structural:lemma}), we will need a little notation. Recall that the family of fibres $\cF$ was defined in~\eqref{def:fibres}, and for each $S \subset V(G)$, define 
\begin{equation}\label{def:tau}
\tau(S) = \min\big\{ t \in \N : \exists \, F_1,\ldots,F_t \in \cF \text{ such that } S \subset F_1 \cup \cdots \cup F_t \big\}.
\end{equation}
In other words, $\tau(S)$ is the smallest number of fibres that cover $S$. A simple but important observation is that the number of $k$-sets $I \subset V(G)$ such that 
\begin{equation}\label{eq:basic:case:for:counting}
\tau\big( I \setminus N_G(v) \big) = O(pm)
\end{equation}
for some $v \in V(G)$ is at most $2^{O(k)} {pn \choose k}$, assuming that $\Delta(G) = O(pn)$ (see Lemma~\ref{lem:counting:c}). 

The basic idea of the structural lemma is that sets $I$ that do not satisfy~\eqref{eq:basic:case:for:counting} for any vertex $v \in V(G)$ should contain many open pairs with respect to either $A$ or $B$. This is not quite true as stated, but would be useful because of the following lemma, which bounds the probability that a set with many open pairs is independent. We say that a subset of $V(G)$ is an $A$-transversal if it contains at most one vertex of each $A$-fibre $F \in \cF_A$, and similarly for $B$. Throughout the paper, we will write $c$ for a small absolute constant (in particular, not depending on $C$), which we will allow to take different values in different lemmas. 

\begin{lemma}\label{lem:manyopen:probbound}
Let $H$ be a graph with $v(H) \le k$, whose vertices are an $A$-transversal, and whose edges satisfy $E(H) \subset O_B$ and $e(H) \ge k^{3/2}$. Then
$$\Pr_A\big( \gaap V(H) \in \cI(G) \text{ and }\, \Delta(A) \le 2pm \gap \big) \le e^{-c p e(H)}.$$
\end{lemma}

Here the notation $\Pr_A$ indicates that the partitions and the graph $B$ are both fixed (and arbitrary), and the probability is only over the randomness of the graph $A$. We will prove Lemma~\ref{lem:manyopen:probbound} in Sections~\ref{sec:inside:edges} and~\ref{sec:outside:edges}. Using it, we can easily bound the expected number of independent $k$-sets that contain an $A$-transversal with at least $\eps k^2$ pairs in $O_B$. Indeed, recalling that $pk^2 = Ck \log k$, it follows from Lemma~\ref{lem:manyopen:probbound} that the expected number of such independent sets is at most ${pn \choose k}$, as long as $C$ is sufficiently large (depending on $c$ and $\eps$). 

Unfortunately, we are not able to prove that every $k$-set either satisfies~\eqref{eq:basic:case:for:counting} for some $v \in V(G)$, or contains such an $A$-transversal, and our structural lemma includes three other cases. Two of these correspond to the existence of a partition $I = U \cup W$ such that $\tau(U) \le (\log k)^{6}$, and moreover either $|U| \ge 2k/3$, or $W$ intersects at most $k/4$ fibres of each partition. In the first of these two cases the counting is easy (see Lemma~\ref{lem:counting:a}), but for the other we will need to choose the partitions $\cF_A$ and $\cF_B$ carefully (see Lemma~\ref{lem:partitions:property}, below). 

The final case is slightly more complicated: we find a vertex $v$ and a (possibly unbalanced) bipartite graph $H$ of $B$-open edges of positive density, such that $\tau\pig( I \setminus N_B(v) \pig)$ is not much larger than the smaller part of $H$. It will turn out (see Lemma~\ref{lem:counting:e}) that, as long as the random graphs $A$ and $B$ both have maximum degree at most $2pm$, we can also bound the expected number of independent $k$-sets with this property using Lemma~\ref{lem:manyopen:probbound}. 

\subsection{Choosing the partitions}

We will count $k$-sets $I \subset V(G)$ of the form $I = U \cup W$, where $\tau(U) \le (\log k)^6$ and $W$ intersects at most $k/4$ fibres of each partition, in Lemma~\ref{lem:counting:b}. In order to do so, we will need to choose the partitions $\cF_A$ and $\cF_B$ so that they satisfy the following simple pseudorandom property. 

\begin{lemma}\label{lem:partitions:property}
We can choose the partitions $\cF_A$ and $\cF_B$ so that, for every $s \in [m]$ and $t \in \N$, there are at most
$$2^{4t} \bigg( \frac{s}{m} \bigg)^{2t - 2s} {n \choose t}$$
$t$-sets $I \subset V(G)$ that intersect at most $s$ fibres of each partition.
\end{lemma}

In the proof of the lemma, and also occasionally later in the paper, it will be convenient to write $\pi_A$ and $\pi_B$ for the natural projections of $V(G)$ onto $V(A)$ and $V(B)$, respectively. 

\begin{proof}[Proof of Lemma~\ref{lem:partitions:property}]
We will choose the partitions $\cF_A$ and $\cF_B$ independently and uniformly at random, and show that the claimed property holds with positive probability. To do so, fix $s$-sets $S \subset V(A)$ and $T \subset V(B)$, and let $X$ be the number of $t$-sets $I \subset V(G)$ such that $\pi_A(I) \subset S$ and $\pi_B(I) \subset T$. Note that the claimed bound holds trivially if $s \ge t$, so we may assume that $s < t$, and observe that
$$\Ex\pig[ X \pig] \le {n \choose t} \bigg( \frac{s}{m} \bigg)^{2t},$$
and hence, by Markov's inequality,
$$\Pr\bigg( X \ge 2^{4t} \bigg( \frac{s}{m} \bigg)^{2t - 2s} {n \choose t} \bigg) \le 2^{-t} \bigg( \frac{s}{em} \bigg)^{2s}.$$
Therefore, taking a union bound over the ${m \choose s}^2$ choices of $S$ and $T$, we deduce that with probability at most
$${m \choose s}^2 2^{-t} \bigg( \frac{s}{em} \bigg)^{2s} < 2^{-t}$$
there are more than the claimed number of such $t$-sets. Summing over $s$ and $t$, we deduce that with positive probability the random partitions have the claimed property for every $s \in [m]$ and $t \in \N$, as required.
\end{proof}

Let us fix partitions $\cF_A$ and $\cF_B$ such that the conclusion of Lemma~\ref{lem:partitions:property} holds. We will not require any other properties of the partitions. 

\subsection{Finding many open pairs}

We have not yet explained how we find an $A$-transversal in $I$ that spans many open pairs. In order to state the lemma that we will use to do so, we first need to define the \emph{spreadness} $\lambda(S)$ of a set $S \subset V(G)$, 
$$\lambda(S) = \max\bigg\{ \frac{|S \cap F|}{|S|} : F \in \cF \bigg\}.$$
In words, $\lambda(S)$ is the largest proportion of $S$ that lies in a single fibre.

Given a set $W \subset V(G)$, we write $o_A(W)$ for the number of $A$-open edges in $W$. Recall that when $v \in V(G)$, we write $N_A(v)$ to denote the neighbourhood of $v$ in the graph $A^*$. We will use the following lemma to find open edges in subsets of $k$-sets that are not effectively covered by a neighbourhood and few fibres. Fix $\eps = 2^{-12}$ for the rest of the paper. 

\begin{lemma}\label{lem:bipartite:lemma}
With high probability, for every pair of disjoint sets $S,T \subset V(G)$ such that
\begin{equation}\label{eq:both:nbhds:big}
\lambda(S) + \lambda(T) \le \frac{1}{(\log m)^2} \qquad \text{and} \qquad \min\bigg\{ \frac{|N_A(v) \cap S|}{|S|}, \frac{|N_A(v) \cap T|}{|T|} \bigg\} \le \eps
\end{equation}
for every $v \in V(G)$, we have
$$o_A(S \cup T) \ge \frac{|S||T|}{32}.$$
\end{lemma}

We will prove Lemma~\ref{lem:bipartite:lemma} in Section~\ref{sec:open:pairs}, and use it in Section~\ref{sec:structural} to prove Lemma~\ref{lem:structural:lemma}. 
To be more precise, we have chosen to write the structural lemma as a deterministic statement, and we therefore need to define an event $\cB$ that says that the conclusion of Lemma~\ref{lem:bipartite:lemma} holds for the random graphs $A$ and $B$. We will also need another event $\cD$, which provides some simple upper bounds on the degree and co-degrees of the vertices of $A$ and $B$. 

\begin{definition}\label{def:events}
Let $\cA = \cB \cap \cD$, where $\cB$ and $\cD$ are defined as follows:
\begin{itemize}
\item $\cB$ is the event that 
$$o_A(S \cup T) \ge \frac{|S||T|}{32}$$
for every pair $S,T \subset V(G)$ of disjoint sets such that~\eqref{eq:both:nbhds:big} holds for every $v \in V(G)$, and that the same also holds with $A$ replaced by $B$.\smallskip
\item $\cD$ is the event that 
$$\max\big\{ \Delta(A), \Delta(B) \big\} \le 2pm \qquad \text{and} \qquad \max\big\{ \Delta_2(A), \Delta_2(B) \big\} \le \log m,$$
where $\Delta_2(H) = \max\big\{ |N_H(u) \cap N_H(v)| : u,v \in V(H), \gap u \ne v \big\}$. 
\end{itemize}
\end{definition}

Observe that the event $\cB$ holds with high probability, by Lemma~\ref{lem:bipartite:lemma}, and that $\cD$ also holds with high probability, since $A$ and $B$ are both copies of $G(m,p)$. 

\begin{lemma}\label{lem:D:holds}
The event $\cD$ holds with high probability. 
\end{lemma}

\begin{proof}
The bound on the maximum degree follows by Chernoff's inequality and the union bound, and the bound on the maximum co-degree holds because $p^2 m \le 1$, which implies that $\Pr\pig( |N_A(u) \cap N_A(v)| \ge t \pig) \le 1/t!$ for every $u,v \in V(A)$ with $u \ne v$.
\end{proof}

\section{Revealing the edges inside a set of size $k$}\label{sec:inside:edges}

In the next two sections we will prove Lemma~\ref{lem:manyopen:probbound}, our bound on the probability that a $k$-set with many open edges with respect to $B$ is independent in $G$. To motivate the~proof, recall from~\eqref{def:G} that if $V(H)$ is independent, then each edge $e \in E(H \cap A^*)$ must be closed with respect to $A$ (since it is in $O_B$), and therefore must have a common neighbour in~$A^*$. Moreover, since $V(H)$ is an $A$-transversal, we may equivalently work in the graph $A$, and bound the probability that each pair in the projection $U \subset V(A)$ that corresponds to an edge of $H$ is either not chosen in $A$, or has a common neighbour in $A$. 

To do so, we will consider separately the common neighbours inside and outside $U$. More precisely, we will first reveal all of the edges of $A$ except those inside $U$, and bound the probability that at least half of the edges of $H$ have a common neighbour outside $U$ (see Lemma~\ref{lem:reveal:outside}). We will then reveal the edges of $A[U]$, and apply the following lemma to the graph of remaining edges of $H$ (those with no common neighbour outside $U$). 

\begin{lemma}\label{lem:good_to_probability}
Let $H$ be a graph with $k$ vertices and $e(H) \ge k^{4/3}$, and let $R \sim G(k,p)$. Then 
$$\Pr\Big( R \cap H \subset K_3(R) \text{ and } \, \Delta(R) \le 1/p \Big) \le e^{-cp e(H)},$$
where $K_3(R)$ denotes the set of edges of $R$ that are contained in a triangle in $R$. 
\end{lemma}

The idea of the proof is similar to that of~\cite[Lemma~3.4]{MSV}, but instead of using an auxiliary random permutation to create independence between our choices, we will use a random subgraph $F \subset H$, and reveal the edges of $R$ in two stages:  the edges of $R' = R - F$ first, and then the edges of $R \cap F$. The key observation is that if $R \cap H \subset K_3(R)$, then with constant probability (in $F$) we have many triangles in $R$ with exactly one edge in $F$. To bound the probability that this happens, we will use the following lemma to bound the number of paths of length two in $R'$, and then Chernoff's inequality to bound the probability that more triangles are completed than expected when we reveal $R \cap F$. 

\begin{lemma}\label{lem:degree-square}
Let $R$ be a random graph with $k$ vertices whose edges are chosen independently, each with probability at most $p$. Then
\begin{equation}\label{eq:degreesquare:prob}
\Pr\bigg( \sum_{v \ina V(R)} d_R(v)^2 \ge M \text{ and }\, \Delta(R) \le 1/p \bigg) \le 2 \cdot e^{-cpM}
\end{equation}
for every $M \ge k^{5/4}$.
\end{lemma}


\begin{proof}
We first claim that, if the event on the left-hand side of~\eqref{eq:degreesquare:prob} holds, then there exist $d,t \in \N$, with
\begin{equation}\label{eq:dt:properties}
\bigg( \frac{M}{8k\log k} \bigg)^{1/2} \le d \le \frac{1}{p} \qquad \text{and} \qquad t \ge \frac{M}{8d^2\log k},
\end{equation}
such that there are at least $t$ vertices in $R$ of degree at least $d$. To see this, partition the vertices of $R$ into sets
$$U_i = \big\{ v \in V(R) : 2^i \le d_R(v) < 2^{i+1} \big\},$$
and observe that if no such pair $(d,t)$ exists, then 
$$|U_i| \le \frac{M}{2^{2i+3}\log k}$$
for every $i \ge 0$ such that $2^i \le 1/p$. Indeed, if $2^{2i} \ge M / 8k\log k$ 
then this follows from~\eqref{eq:dt:properties} with $d = 2^i$, and otherwise it follows from the trivial bound $|U_i| \le k$. 
It follows that either $\Delta(R) > 1/p$, or 
$$\sum_{v \ina V(R)} d_R(v)^2 \le \sum_{i \eqa 0}^{\log_2(1/p)} \sum_{v \ina U_i} 2^{2i+2} \le \sum_{i \eqa 0}^{\log_2(1/p)} \frac{M}{2\log k} < M.$$
Since each edge is included in $R$ independently with probability at most $p$, 
the probability that there are $t$ vertices of degree at least $d$ is at most
\begin{equation}\label{eq:degreesquare:calc}
{k \choose t} {tk \choose td/2} p^{td/2} \le \bigg( k \cdot \bigg( \frac{2epk}{d} \bigg)^{d/2} \bigg)^t \le \exp\bigg( - \frac{dt\log k}{2^5} \bigg) \le \exp\left(-\frac{M}{2^8 d}\right),
\end{equation}
where the second inequality holds because $pk = C \log k$ and $d \gg k^{1/8} / \log k \gg k^{1/16} \log k$, 
by~\eqref{eq:dt:properties}, and our lower bound $M \ge k^{5/4}$, 
and the third follows from~\eqref{eq:dt:properties}. Summing over the choices of $d$ and $t$ satisfying~\eqref{eq:dt:properties}, and recalling that $d \le 1/p$, 
the claimed bound follows. 
\end{proof}

We are now ready to prove Lemma~\ref{lem:good_to_probability}.

\begin{proof}[Proof of Lemma~\ref{lem:good_to_probability}]
Let $M = e(H)$, and note that $e(R \cap H) \sim \Bin(M,p)$, so 
\begin{equation}\label{eq:chernoff:RcapH}
\Pr\big( e(R \cap H) \le pM/2 \big) \le e^{-pM/8},
\end{equation}
by Chernoff's inequality. It will therefore suffice to bound the probability of the event
$$\cR = \big\{ R \cap H \subset K_3(R) \big\} \cap \big\{ \Delta(R) \le 1/p \big\} \cap \big\{ e(R \cap H) \ge pM/2 \big\}.$$
In order to do so, we will choose a random subgraph $F \subset H$, where each edge of $H$ is included in $F$ independently at random with probability $\alpha = 2^{-5}$, and define 
$$Y = \big| \big\{ xy \in E(R \cap F) : \exists \, z \in V(R) \text{ such that } xz,yz \in E(R - F) \big\} \big|$$
to be the number of edges of $R \cap F$ that are the unique edge of $F$ in some triangle in $R$. 

The main step in the proof is the following claim. 

\begin{claim}\label{claim:Ubound:W}
\begin{equation}\label{eq:Ubound:W}
\Pr(\gap\cR) \le 2 \cdot \Pr\big( \gap Y \ge 2^{-8} pM \text{ and } \, \Delta(R) \le 1/p \gap \big).
\end{equation}
\end{claim} 

\begin{proof}
To prove this claim, recall that if $R_0 \in \cR$, then for each edge $e \in E(R_0 \cap H)$, there exists a triangle $T_e \subset R_0$ with $e \in E(T_e)$. For each edge fix such a triangle, define 
$$X(R_0) = e(R_0 \cap F) \qquad \text{and} \qquad B(R_0) = \big|\big\{ e \in E(R_0 \cap F) : e(T_e \cap F) \ge 2 \big\} \big|,$$
and observe that if $R = R_0$, then $Y \ge X(R_0) - B(R_0)$. Moreover, we have 
$$\Pr_F\bigg( X(R_0) \le \gap \frac{\alpha}{2} \cdot e(R_0 \cap H) \bigg) \le \frac{1}{4} \qquad \text{and} \qquad \Ex_F\pig[ B(R_0) \pig] \le 2\alpha^2 \cdot e(R_0 \cap H),$$
since $X(R_0) \sim \Bin\pig( e(R_0 \cap H), \alpha \pig)$ and $e(R_0 \cap H) \ge pM/2$ is sufficiently large. 

By Markov's inequality and our choice of $\alpha$, it follows that
$$\Pr\bigg( Y \ge \gap \frac{\alpha}{4} \cdot e(R \cap H) \;\Big|\; R = R_0 \bigg) \ge \frac{1}{2}$$
for every $R_0 \in \cR$. Finally, averaging over the choice of $R_0$, we obtain  
$$\Pr\bigg( \cR \cap \bigg\{ Y \ge \gap \frac{\alpha}{4} \cdot e(R \cap H) \bigg\} \bigg) \ge \frac{1}{2} \cdot \Pr(\cR).$$
Since the event $\cR$ implies that $e(R \cap H) \ge pM/2$ and $\Delta(R) \le 1/p$, the claim follows.
\end{proof}

To bound the right-hand side of~\eqref{eq:Ubound:W}, we apply Lemma~\ref{lem:degree-square} to the graph $R' = R - F$, and Chernoff's inequality to the graph $R \cap F$, for each choice of $F \subset H$. To be precise, for any fixed $F$ we have
\begin{equation}\label{eq:degree-square:app}
\Pr\bigg( \sum_{v \ina V(R')} d_{R'}(v)^2 \ge 2^{-9} M \text{ and } \Delta(R') \le 1/p \bigg) \le e^{-cpM}
\end{equation}
for some constant $c > 0$, by Lemma~\ref{lem:degree-square}. On the other hand, observe that given $F$ and $R'$, the conditional distribution of $Y$ is $\Bin(t,p)$, 
where $t$ is the number of edges of $F$ with a common neighbour in $R'$. It follows that, for any $R_0$ such that $\sum_{v \ina V(R_0)} d_{R_0}(v)^2 < 2^{-9} M$ and any choice of $F$, we have 
$$\Pr\big( \gaap Y \ge 2^{-8} pM \,\big|\, R' = R_0 \gap \big) \le \Pr\big( \Bin\pig( 2^{-9} M, p \pig) \ge 2^{-8} pM \big) \le e^{-cpM}$$
by Chernoff's inequality. Averaging over the choices of $F$ and $R_0$, and recalling~\eqref{eq:chernoff:RcapH}, ~\eqref{eq:degree-square:app} and Claim~\ref{claim:Ubound:W} it follows that  
$$\Pr\Big( R \cap H \subset K_3(R) \text{ and } \gaaap \Delta(R) \le 1/p \Big) \le \gap \Pr(\gap\cR) + e^{-pM/8} \le 5 \cdot e^{-cpM}.$$
Adjusting the constant $c$ slightly, this completes the proof of the lemma. 
\end{proof}

\begin{remark}
The proof above implies more generally that if $R \sim G(N,q)$ and $H$ is a graph with $N$ vertices and $e(H) \gg \max\big\{ N \log N, q^2N^{3+\gamma} \big\}$ edges for some $\gamma > 0$, then 
$$\Pr\Big( R \cap H \subset K_3(R) \text{ and } \, \Delta(R) \le 1/q \Big) \le 2 \cdot e^{-cq e(H)}$$
for some constant $c > 0$ depending on $\gamma$. Indeed, the bound on $e(H)$ implies that $d \gg 1$ and $d \ge qN^{1+\gamma/3}$, which suffice for the second inequality in~\eqref{eq:degreesquare:calc}, and the remaining steps of the proof only require the weaker bound $e(H) \gg 1/q$.  
\end{remark}

\section{The proof of Lemma~\ref{lem:manyopen:probbound}}\label{sec:outside:edges}

To deduce Lemma~\ref{lem:manyopen:probbound} from Lemma~\ref{lem:good_to_probability}, we will reveal the edges of $A$ in two stages. First we will reveal the edges outside the set $U$, and apply the following lemma to bound the probability that more than $3/4$ of the edges of $H$ have a common neighbour in $A$ outside $U$. We will then apply Lemma~\ref{lem:good_to_probability} to the surviving edges of $H$. 

\begin{lemma}\label{lem:reveal:outside}
Let $H$ be a bipartite graph with $V(H) \subset V(A)$ and $v(H) \le v(A)/2$, and set
$$Z = \big| \big\{ xy \in E(H) : \exists \, z \in V(A) \setminus V(H) \text{ such that } xz,yz \in E(A) \big\} \big|.$$
Then
\begin{equation}\label{eq:gout}
\Pr\bigg( Z \ge \frac{e(H)}{2} \gap\text{ and }\, \Delta(A) \le 2pm \bigg) \le \exp\bigg( - \frac{p \gap e(H)}{4} \bigg).
\end{equation}
\end{lemma}

The proof of this lemma is quite simple: let $S$ and $T$ be the parts of $H$, so $V(H) = S \cup T$, reveal the edges of $A$ between $S$ and $W = V(A) \setminus V(H)$, and then apply the following corollary of Chernoff's inequality to bound the probability that $Z$ is larger than expected when we reveal the edges between $T$ and $W$. 

\begin{chernoff}
Let $0 \le a_1,\ldots,a_N \le M$ and let $X_1,\ldots,X_N$ be independent $\Ber(p)$ random variables. Then
$$\Pr\bigg( \sum_{i = 1}^N a_i X_i \ge t \bigg) \le \exp\bigg( - \frac{t}{8M}  \bigg)$$
for every $t \ge 2p \sum_{i = 1}^N a_i$. 
\end{chernoff}  

We will apply this inequality with 
$$\big\{ a_1,\ldots,a_N \big\} = \big\{ |N_H(v) \cap N_A(w) \cap S| : v \in T, \, w \in W \big\},$$
so that the sum of the $a_i$ that correspond to edges $vw \in E(A)$ is equal to the number of triples $uvw$ with $uv \in E(H[S,T])$ and $uw,vw \in E(A)$, and hence is an upper bound on $Z$.

\begin{proof}[Proof of Lemma~\ref{lem:reveal:outside}]
Let $S$ and $T$ be the parts of $H$, and set $W = V(A) \setminus V(H)$. Reveal the edges of $A[S,W]$, and assume that $\Delta\pig( A[S,W] \pig) \le 2pm$, since otherwise the event on the left-hand side of~\eqref{eq:gout} cannot hold. Now, for each $v \in T$ and $w \in W$, define 
$$a(v,w) = |N_H(v) \cap N_A(w) \cap S|,$$
and observe that 
$$Z \le \sum_{v \ina T} \sum_{w \ina W} a(v,w) \1\big[ vw \in E(A) \big].$$
Note also that $a(v,w) \le \Delta\pig( A[S,W] \pig) \le 2pm$, and that
$$2p \sum_{v \ina T} \sum_{w \ina W} a(v,w) \le 2p \cdot \Delta\pig( A[S,W] \pig) \cdot e(H) \le 4p^2m \cdot e(H) \le \frac{e(H)}{2},$$
since $p^2 m \le 1/8$. Therefore, applying Chernoff's inequality with $t = e(H)/2$ and $M = 2pm$ to the random choice of the edges of $A[T,W]$ and the sequence $a(v,w)$, we deduce that
$$\Pr\bigg( Z \ge \frac{e(H)}{2} \bigg) \le \exp\bigg( - \frac{e(H)}{32pm} \bigg) \le \exp\bigg( - \frac{p \gap e(H)}{4} \bigg),$$
as claimed, since $p^2 m \le 1/8$. 
\end{proof}

Combining Lemmas~\ref{lem:good_to_probability} and~\ref{lem:reveal:outside} gives Lemma~\ref{lem:manyopen:probbound}.

\begin{proof}[Proof of Lemma~\ref{lem:manyopen:probbound}]
Since $V(H)$ is an $A$-transversal, each edge of $H$ corresponds to a distinct pair of vertices of $U = \pi_A(V(H)) \subset V(A)$. Reveal the edges of $A$ between $U$ and $V(A) \setminus U$, and define $H'$ to be the graph with edge set
$$E(H') = \big\{ xy \in E(H) : \nexists \, z \in V(A) \setminus U \text{ such that } xz,yz \in E(A) \big\}.$$
By Lemma~\ref{lem:reveal:outside}, applied to a bipartite subgraph of $H$ with at least $e(H)/2$ edges, we have
\begin{equation}\label{eq:surviving:edges}
\Pr\bigg( e(H') < \frac{e(H)}{4} \gap\text{ and }\gaaap \Delta(A) \le 2pm \bigg) \le \exp\bigg( - \frac{pe(H)}{8} \bigg).
\end{equation}
We now apply Lemma~\ref{lem:good_to_probability} to the graphs $R = A[U]$ and $H'$, 
which gives
$$\Pr\Big( A[U] \cap H' \subset K_3(A[U]) \text{ and } \gaaap \Delta(A[U]) \le 1/p \Big) \le 2 \cdot e^{-cp e(H)/4}.$$
Finally, observe that if $V(H) \in \cI(G)$ and $\Delta(A) \le 2pm$, then every edge of $A \cap H$ must be contained in a triangle of $A$, and $\Delta(A[U]) \le \Delta(A) \le 2pm \le 1/p$. Thus, combining these two inequalities, it follows that
$$\Pr_A\big( \gaap V(H) \in \cI(G) \text{ and }\gaaap \Delta(A) \le 2pm \gap \big) \le 2 \cdot e^{-c p e(H)/8},$$
as required.
\end{proof}

\section{Bounding the number of open pairs}\label{sec:open:pairs}

In this section we will prove Lemma~\ref{lem:bipartite:lemma}. The main step will be the following lemma; we will then deduce the statement for well-spread sets in $G$ via a simple averaging argument. Given disjoint sets $S,T \subset V(A)$, define 
$$o_A(S,T) = \big| \big\{ (u,v) \in S \times T : N_A(u) \cap N_A(v) = \emptyset \big\} \big|$$
to be the number of open edges with respect to $A$ with one vertex in $S$ and the other in $T$. Recall that $C$ is a sufficiently large constant, and that $\eps = 2^{-12}$.  

\begin{lemma}\label{lem:bipartite:A}
With high probability, for every pair $S,T \subset V(A)$ of disjoint sets such that
\begin{equation}\label{eq:ST:sizes}
\frac{C\log m}{2} \le |S| \le |T| \le C \log m
\end{equation}
and  
\begin{equation}\label{eq:both:nbhds:bigger}
\min\bigg\{ \frac{|N_A(v) \cap S|}{|S|}, \frac{|N_A(v) \cap T|}{|T|} \bigg\} \le 4\eps
\end{equation}
for every $v \in V(A)$, we have
$$o_A(S,T) \ge \frac{|S||T|}{4}.$$
\end{lemma}

The proof of this lemma is again quite straightforward. We first use the property~\eqref{eq:both:nbhds:bigger} and a simple first moment argument to bound the number of pairs in $S \times T$ that are closed by vertices with at least two neighbours in $S$. The point is that there is a large set of vertices (larger than $S \cup T$ by a factor of $1/\eps$) that each have at least three neighbours in $S \cup T$. 

\pagebreak

To bound the number of pairs closed by vertices with exactly one neighbour in $S$, on the other hand, we will first reveal the edges between $S$ and $W = V(A) \setminus(S \cup T)$, and show that with (very) high probability each vertex $u \in S$ has a set $Q(u)$ of roughly $pm$ unique neighbours. We will then use Chernoff's inequality to control the number of pairs $(u,v) \in S \times T$ such that $N_A(v) \cap Q(u) \ne \emptyset$ when we reveal the edges between $T$ and $W$. 

\begin{proof}[Proof of Lemma~\ref{lem:bipartite:A}]
Fix disjoint sets $S,T \subset V(A)$ satisfying~\eqref{eq:ST:sizes}, and let $\cN$ be the event that~\eqref{eq:both:nbhds:bigger} holds for every $v \in V(A)$. We will show that
$$\Pr\Big( \cN \cap \big\{ o_A(S,T) < |S||T|/4 \big\} \Big) \le m^{-3C \log m}.$$
The lemma will then follow, simply by taking a union bound over all choices of $S$ and $T$ with sizes in the given range.

The idea of the proof is to consider separately those pairs that are closed by a  vertex in the set 
$$Q = \big\{ w \in W : |N_A(w) \cap S| = 1 \big\},$$
where $W = V(A) \setminus(S \cup T)$, and those that are closed by vertices in $V(A) \setminus Q$. To control the number of pairs of the second type, write 
$$d_S(w) = |N_A(w) \cap S| \qquad \text{and} \qquad d_T(w) = |N_A(w) \cap T|,$$
and note that $w$ closes exactly $d_S(w)d_T(w)$ pairs. The following claim therefore provides a suitable bound on the probability that there are many pairs of the second type. 

\begin{claim}\label{claim:outside}
\begin{equation}\label{eq:closed:outside:Q}
\Pr\bigg( \cN \cap \bigg\{ \sum_{w \ina V(A) \setminusa Q} d_S(w) d_T(w) \ge \frac{|S| |T|}{4} \bigg\} \bigg) \le \exp\bigg( - \frac{|S| \log m}{2^8 \eps} \bigg).
\end{equation}
\end{claim}

\begin{proof}
Observe first that $d_S(w) d_T(w) = 0$ if $w \in P \setminus Q$, where
$$P = \big\{ w \in W : |N_A(w) \cap (S \cup T)| \le 2 \big\}.$$ 
Recall also that if the event $\cN$ holds, then either $d_S(w) \le 4\eps |S| \le 4\eps |T|$ or $d_T(w) \le 4\eps |T|$, and therefore 
$$d_S(w) d_T(w) \le 4\eps |T| \big( d_S(w) + d_T(w) \big),$$
for each $w \in V(A)$. The event on the left-hand side in~\eqref{eq:closed:outside:Q} therefore implies that 
\begin{equation}\label{eq:product:sum}
\sum_{w \ina V(A) \setminusa P} \pig| N_A(w) \cap (S \cup T) \pig| \ge \frac{|S|}{16\eps}.
\end{equation}
We claim that there exists a set $Z \subset W \setminus P$ such that
$$|Z| \le \frac{t}{3} \qquad \text{and} \qquad e\big( A[S \cup T \cup Z] \big) \ge t,$$
where $t = |S|/32\eps$. Indeed, we can construct such a set greedily by adding vertices of $W \setminus P$ one by one to the set $S \cup T$, since each has at least three neighbours in $S \cup T$. Taking a union bound over the choices of $Z$, and recalling that $p = \Theta\pig( m^{-1/2} \pig)$, it follows that
$$\Pr\bigg( \sum_{w \ina V(A) \setminusa P} \pig| N_A(w) \cap (S \cup T) \pig| \ge \frac{|S|}{16\eps} \gap \bigg) \le {m \choose t/3} {t^2 \choose t} p^t \le \big( p^3 m  \pigl( \log m \pig)^4 \big)^{t/3} \le m^{-t/8},$$
as claimed. 
\end{proof}

It remains to bound the number of pairs in $S \times T$ that have a common neighbour in $Q$. To do so, we start by revealing all of the edges between $S$ and $W$. For each $u \in S$, let
$$Q(u) = \big\{ w \in W : N_A(w) \cap S = \{u\} \big\}$$
be the set of unique neighbours of $u$ with respect to $S$. Let $\cQ$ denote the event that  
$$|Q(u)| \le 2pm$$
for every $u \in S$, and observe that $\cQ$ holds with probability at least $1 - \exp\pig( - m^{1/4} \pig)$, by Chernoff's inequality, since $Q(u) \subset N_A(u)$ and $p = \Theta\pig( m^{-1/2} \pig)$. 

Now, observe that $uv$ is closed by a vertex in $Q$ if and only if $N_A(v) \cap Q(u)$ is non-empty. Since the sets $Q(u)$ are disjoint, these events are independent for each $(u,v) \in S \times T$. The following claim therefore follows easily from Chernoff's inequality. 

\begin{claim}\label{claim:Q:closers}
$$\Pr\bigg( \cQ \cap \bigg\{ \big|\big\{ (u,v) \in S \times T : N_A(v) \cap Q(u) = \emptyset \big\} \big| \le \frac{|S||T|}{2} \bigg\} \bigg) \le e^{-c|S||T|}$$
for some constant $c > 0$. 
\end{claim}

\begin{proof}
Recall that the event $\cQ$ depends only on the edges between $S$ and $W$, so we may reveal these edges and assume that $\cQ$ holds. For any such choice of $A[S,W]$, the events $N_A(v) \cap Q(u) = \emptyset$ are independent for all $(u,v) \in S \times T$, and have probability at least 
$$(1 - p)^{|Q(u)|} \ge \exp\big( - 4p^2 m \big) \ge e^{-1/2} > \frac{1}{2}$$
since $\cQ$ holds and $p^2 m \le 1/8$. The probability that at most $|S||T|/2$ of these events occur is thus at most $e^{-c|S||T|}$ for some constant $c > 0$, as claimed, by Chernoff's inequality. 
\end{proof}

Combining Claims~\ref{claim:outside} and~\ref{claim:Q:closers}, it follows that
$$\Pr\bigg( \cN \cap \bigg\{ o_A(S,T) < \frac{|S||T|}{4} \bigg\} \bigg) \le e^{- m^{1/4}} + e^{-c|S||T|} + \exp\bigg( - \frac{|S| \log m}{2^8 \eps} \bigg) \le m^{-3C \log m}$$
if $m$ is sufficiently large, $C \ge 16/c$ and $\eps = 2^{-12}$, since $|T| \ge |S| \ge (C/2) \log m$. As noted above, the lemma now follows by taking a union bound over the choices of $S$ and $T$.
\end{proof}

In order to deduce Lemma~\ref{lem:bipartite:lemma} from Lemma~\ref{lem:bipartite:A}, we will consider random subsets $X \subset S$ and $Y \subset T$ of size $C \log m$, and show that with probability at least $1/2$ we can apply the latter lemma to disjoint subsets of $\pi_A(X)$ and $\pi_A(Y)$. 

\begin{proof}[Proof of Lemma~\ref{lem:bipartite:lemma}]
Let $X \subset S$ and $Y \subset T$ be random subsets of size $|X| = |Y| = C \log m$, and observe that
$$\Ex\big[ o_A(X,Y) \big] = \frac{|X||Y|}{|S||T|} \cdot o_A(S,T).$$
To prove the lemma, it will therefore suffice to show that
$$\Pr\bigg( o_A(X,Y) \ge \frac{|X||Y|}{16} \bigg) \ge \frac{1}{2}.$$
To do so, we will apply Lemma~\ref{lem:bipartite:A} to disjoint sets $X' \subset \pi_A(X)$ and $Y' \subset \pi_A(Y)$. To bound the size of these sets, let
$$Z = \big\{ \{ x,y \} \subset X \cup Y : \pi_A(x) = \pi_A(y) \big\},$$
and observe that 
$$\Ex\pig[ |Z| \pig] \le \pig( \lambda(S) + \lambda(T) \pig) \pig( |X| + |Y| \pig)^2 = O(1)$$
by our bound on $\lambda(S) + \lambda(T)$. By Markov's inequality, it follows that with probability at least $3/4$, there exist such sets $X'$ and $Y'$ with 
$$\frac{C\log m}{2} \le |X'| \le |Y'| \le C \log m.$$
Moreover, by Chernoff's inequality and~\eqref{eq:both:nbhds:big}, we have
$$\Pr\bigg( \min\bigg\{ \frac{|N_A(v) \cap X|}{|X|}, \frac{|N_A(v) \cap Y|}{|Y|} \bigg\} \ge 2\eps \bigg) \le e^{-\eps^2 |X|} \le \frac{1}{m^2}$$
for every $v \in V(G)$, since $C$ is sufficiently large. Thus, taking a union bound over the choice of $v$, 
it follows that the neighbourhood condition $\cN$ holds for the pair $(X',Y')$ with probability at least $3/4$. Hence, applying  Lemma~\ref{lem:bipartite:A} to the sets $X'$ and $Y'$ whenever they satisfy the conditions of the lemma, we see that 
$$\Pr\bigg( o_A(X,Y) \ge \frac{|X'||Y'|}{4} \bigg) \ge \frac{1}{2},$$
as required.
\end{proof}

\section{The Structural Lemma}\label{sec:structural}

The final missing piece in the proof of Theorem~\ref{thm:acmgraph} is the following technical lemma, which divides the family of independent $k$-sets into five families. We will bound the expected number of independent sets in each of these families in Section~\ref{final:proof:sec}, using Lemma~\ref{lem:manyopen:probbound} and the pseudorandom 
property of the partitions $\cF_A$ and $\cF_B$ given by Lemma~\ref{lem:partitions:property}. 

Recall from Definition~\ref{def:events} that the event $\cA$ holds if the conclusion of Lemma~\ref{lem:bipartite:lemma} holds for both $A$ and $B$, and the degrees and co-degrees of both graphs are all not too much larger than expected. Recall also from~\eqref{def:tau} that $\tau(U)$ is the smallest number of fibres $F \in \cF$ that cover $U$, and for each $X \in \{A,B\}$, let us write $\tau_X(U) = |\pi_X(U)|$ for the number of fibres $F \in \cF_X$ that intersect $U$. In the statement of the following lemma $c > 0$ is a sufficiently small absolute constant (depending on $\eps$ but not on $C$), and we set $\{X,Y\} = \{A,B\}$, so if $X = A$ then $Y = B$, and vice versa.

\begin{lemma}\label{lem:structural:lemma} 
If the event $\cA$ holds, then for every set $I \subset V(G)$ with $|I| = k$, one of the following holds:
\begin{itemize}
\item[$(a)$] There is a set $U \subset I$ with $|U| \ge 2k/3$ and $\tau(U) \le (\log k)^6$.\smallskip 
\item[$(b)$] There is a set $W \subset I$ with $|W| \ge k/3$ such that
$$\tau_A(W) + \tau_B(W) \le \frac{k}{4}  \qquad \text{and} \qquad \tau\pig( I \setminus W \pig) \le (\log k)^6.$$
\item[$(c)$] There exists a vertex $v \in V(G)$ such that 
$$\tau\big( I \setminus N_{G^*}(v) \big) \le 3pm.$$
\item[$(d)$] There exist $W \subset I$ and $X \in \{A,B\}$ such that $W$ is an $X$-transversal, and
$$o_Y(W) \ge ck^2.$$
\item[$(e)$] There exists $X \in \{A,B\}$ and disjoint sets $S,T \subset I$, with
$$|S| \ge pm \qquad \text{and} \qquad |T| \ge ck,$$ 
such that $S \cup T$ is an $X$-transversal, and moreover 
$$o_Y(S \cup T) \ge \frac{|S||T|}{32} \qquad \text{and} \qquad \tau\pig( I \setminus N_Y(v) \pig) \le 3|S|$$
for some $v \in V(G)$.
\end{itemize}  
\end{lemma}

We will use the following easy observation in the proof of Lemma~\ref{lem:structural:lemma}.

\begin{observation}\label{obs:spread:and:fewfibes}
Let $I \subset V(G)$ with $|I| \le k$, and let $\lambda > 0$. There exists a partition $I = U \cup W$ such that 
$$\tau(U) \le \frac{\log k}{\lambda} + 1 \qquad \text{and} \qquad \lambda(W) \le \lambda.$$
\end{observation}

\begin{proof}
Starting with $W = I$ and $U = \emptyset$, we repeatedly move sets of the form $F \cap W$ for some fibre $F \in \cF$ from $W$ to $U$, until we obtain a suitable partition of $I$. To be precise, recall that if $\lambda(W) > \lambda$, then there exists a fibre $F$ with $|F \cap W| \ge \lambda |W|$. Move the elements of this set from $W$ to $U$, and repeat until $\lambda(W) \le \lambda$. Noting that $\lambda(\emptyset) = 0$, if there are $t+1$ such steps in total, then 
$$\tau(U) \le t + 1 \qquad \text{and} \qquad 1 \le |W| \le (1 - \lambda)^t |I| \le e^{- \lambda t} k$$
after $t$ steps, so $t \le \lambda^{-1} \log k$, as claimed. 
\end{proof}

We can now prove the structural lemma.

\begin{proof}[Proof of Lemma~\ref{lem:structural:lemma}]
Let $I = U \cup W$ be the partition given by Observation~\ref{obs:spread:and:fewfibes} applied to the set $I$ with $\lambda = (\log m)^{-4}$, and note that    
$$\tau(U) \le (\log k)^6 \qquad \text{and} \qquad \lambda(W) \le \frac{1}{(\log m)^4},$$ 
since $m \le k^2$. If $|U| \ge 2k/3$, then we are in case $(a)$, so we may assume that $|W| \ge k/3$. We are then in case $(b)$ if $\tau_A(W) + \tau_B(W) \le k/4$, so we may assume that $\tau_A(W) \ge k/8$. 

\begin{claim}\label{claim:structural:e}
Either there exists $v \in V(G)$ such that
\begin{equation}\label{eq:structural:e}
\tau_A\big( N_B(v) \cap W \big) \ge 2ck,
\end{equation}
or we are in case $(d)$. 
\end{claim}

\begin{proof}[Proof of Claim~\ref{claim:structural:e}]
Since $\tau_A(W) \ge k/8$, we may choose disjoint sets $S,T \subset W$ such that $S \cup T$ is an $A$-transversal and $|S| = |T| = k/16$. We claim that if~\eqref{eq:structural:e} does not hold, then the conditions of property~$\cB$ (that is, of Lemma~\ref{lem:bipartite:lemma}) 
are satisfied by the pair $(S,T)$. 

To see this, observe first that, since $\lambda(W) \le (\log m)^{-3}$, we have 
$$|S \cap F| \le |W \cap F| \le \frac{|W|}{(\log m)^3} \le \frac{16|S|}{(\log m)^3} \ll \frac{|S|}{(\log m)^2},$$
for every fibre $F \in \cF$, so $\lambda(S) \ll (\log m)^{-2}$, and similarly for $T$. Moreover, for each $v \in V(G)$, if~\eqref{eq:structural:e} does not hold, then
$$\frac{|N_B(v) \cap S|}{|S|} = \frac{\tau_A\pig( N_B(v) \cap S \pig)}{k/16} \le \frac{\tau_A\pig( N_B(v) \cap W \pig)}{k/16} \le 32c \le \eps,$$
since $S$ is an $A$-transversal and $c$ is sufficiently small. Since $\cB$ holds, it follows that 
$$o_B(S \cup T) \ge \frac{|S||T|}{32} \ge \frac{k^2}{2^{13}},$$
and hence we are in case $(d)$, as claimed.
\end{proof}

We may therefore assume that~\eqref{eq:structural:e} holds for some $v \in V(G)$. Set
$$Q = W \setminus N_B(v),$$
and observe that if $\tau_A(Q) \le 2pm$, then we are in case $(c)$, since 
$$I \setminus N_{G^*}(v) \subset Q \cup U \qquad \text{and} \qquad \tau(Q) + \tau(U) \le \tau_A(Q) + (\log k)^6 \le 3pm.$$ 
It therefore remains to show that if $\tau_A(Q) \ge 2pm$, then we are in case $(e)$.  

To do so, let us first choose disjoint sets $S \subset Q$ and $T \subset N_B(v) \cap W$ with
$$|S| \ge \frac{\tau_A(Q)}{2} \ge pm \qquad \text{and} \qquad |T| \ge \frac{\tau_A\pig( N_B(v) \cap W \pig)}{2} \ge ck$$
such that $S \cup T$ is an $A$-transversal, and observe that
$$\tau\pig( I \setminus N_B(v) \pig) \le \tau_A(Q) + \tau(U) \le 2|S| + (\log k)^6 \le 3|S|.$$
It will therefore suffice to show that $o_B(S \cup T) \ge |S||T|/32$. Since the event $\cB$ holds, to do so we only need to check that the sets $S$ and $T$ satisfy the conditions of Lemma~\ref{lem:bipartite:lemma}. To see this, observe first that
$$|S \cap F| \le |W \cap F| \le \frac{|W|}{(\log m)^4} \ll \frac{|S|}{(\log m)^2},$$
for every fibre $F \in \cF$, since $\lambda(W) \le (\log m)^{-4}$ and $|S| \ge pm = \Theta\pig( k/\log k \pig)$, and
$$|T \cap F| \le |W \cap F| \le \frac{|W|}{(\log m)^4} \ll \frac{|T|}{(\log m)^2}$$
since $|T| \ge ck$. Finally, observe that if $u \in V(G)$ and $\pi_B(u) \ne \pi_B(v)$, then
$$\frac{|N_B(u) \cap T|}{|T|} \le \frac{|N_B(u) \cap N_B(v) \cap W|}{ck} \le \frac{\lambda(W) \cdot \Delta_2(B)}{c} \le \frac{1}{\log m}$$
since $T \subset N_B(v) \cap W$, the event $\cD$ holds (so $\Delta_2(B) \le \log m$), and $\lambda(W) \le (\log m)^{-3}$. On the other hand, if $\pi_B(u) = \pi_B(v)$, then
$$|N_B(u) \cap S| = 0,$$
since $S \subset W \setminus N_B(v) = W \setminus N_B(u)$. It follows that the conditions of property $\cB$ are satisfied, and hence we have $o_B(S \cup T) \ge |S||T|/32$, as required.
\end{proof}

\section{Proof of Theorem~\ref{thm:acmgraph}}\label{final:proof:sec}

In this final section we will complete the proof of Theorem~\ref{thm:acmgraph} by showing that, with high probability, the random graph $G$ constructed in Section~\ref{sec:construction} has the claimed number of independent sets of size $k$. To do so, we will consider separately each of the five cases in Lemma~\ref{lem:structural:lemma}; in cases $(a)$--$(c)$ the counting will be deterministic, while in cases $(d)$ and $(e)$ we will use Lemma~\ref{lem:manyopen:probbound} to bound the expected number of independent sets with the given properties. We begin with case~$(a)$, where the counting is very easy. 

\begin{lemma}\label{lem:counting:a}
There are at most ${pn \choose k}$ sets $I \subset V(G)$ with $|I| = k$ that contain a set $U$ with 
$$|U| \ge \frac{2k}{3} \qquad \text{and} \qquad \tau(U) \le (\log k)^6.$$
\end{lemma}

\begin{proof}
We first choose $\ell = (\log k)^6$ fibres, then $2k/3$ vertices contained in these fibres, and then $k/3$ additional vertices. The number of choices is thus at most
$${2m \choose \ell} {\ell n / m \choose 2k/3} {n \choose k/3} \le 2^{O(k)} \bigg( \frac{p^4 n^2 (\log k)^{12}}{k^2} \cdot \frac{n}{k} \bigg)^{k/3} \le {pn \choose k},$$
as claimed, since $m = \Theta\pig( p^{-2} \pig)$, $k \gg \ell \log k$ and $p \le (\log k)^{-13}$. 
\end{proof}

We next deal with case~$(c)$, which is also easy; note that there are at most $m^2$ choices for the neighbourhood $N_{G^*}(v)$, since it is determined by the fibres containing $v$. 

\begin{lemma}\label{lem:counting:c}
If the event $\gap\cD$ holds, then for each $v \in V(G)$, there are at most 
$$2^{5k} {pn \choose k}$$ 
$k$-sets $I \subset V(G)$ with $\tau\big( I \setminus N_{G^*}(v) \big) \le 4pm$.
\end{lemma}

\begin{proof} 
We need to choose $t$ vertices of $N_{G^*}(v)$, $4pm$ fibres, and $k-t$ vertices contained in these fibres. Since $4pm = (2p)^{-1} \le k / \log k$, the number of choices is at most
$${2m \choose 4pm} \sum_{t = 0}^k {4pn \choose t} {4pn \choose k - t} \le 2^{k}\binom{8pn}{k} \le 2^{5k} {pn \choose k},$$
as claimed.
\end{proof}

For case~$(b)$, we will need the property of the random partitions that was guaranteed in Lemma~\ref{lem:partitions:property}. 

\begin{lemma}\label{lem:counting:b}
There are at most ${pn \choose k}$ sets $I \subset V(G)$ with $|I| = k$ that contain a set $W$ with 
$$|W| \ge \frac{k}{3}, \qquad \tau_A(W) + \tau_B(W) \le \frac{k}{4} \qquad \text{and} \qquad \tau\pig( I \setminus W \pig) \le (\log k)^6.$$ 
\end{lemma}

\begin{proof}
We need to choose $\ell = (\log k)^6$ fibres covering $I \setminus W$, a set of $t \le 2k/3$ elements of these fibres, and a set $W$ of size $k - t$ that intersects at most $k/4$ fibres of each partition. By Lemma~\ref{lem:partitions:property}, the number of choices is at most
$${2m \choose \ell} \sum_{t \eqa 0}^{2k/3} {\ell n / m \choose t} 2^{4(k-t)} \bigg( \frac{k}{4m} \bigg)^{2(k - t) - k/2} {n \choose k - t}.$$
Since $m = \Theta\pig( p^{-2} \pig)$ and $k \gg \ell \log k$, this is at most
$$2^{O(k)} \sum_{t \eqa 0}^{2k/3} \bigg( \frac{p^2 n (\log k)^6}{t} \cdot \frac{1}{p^4k^2} \cdot \frac{k}{n} \bigg)^t \pig( p^2k \pig)^{3k/2} {n \choose k},$$
which, since $pk = \Theta\pigl( \log k \pig)$ and $p \le (\log k)^{-13}$, is at most
$$2^{O(k)} \sum_{t \eqa 0}^{2k/3} \bigg( \frac{(\log k)^5}{pt} \bigg)^t p^{k/2} \pigl( \log k \pig)^{3k/2} {pn \choose k} \le {pn \choose k},$$
as claimed.
\end{proof}

For cases $(d)$ and $(e)$ we will need to use the randomness of $A$. We first deal with case $(d)$, which is somewhat easier. 

\begin{lemma}\label{lem:counting:d}
With high probability there are at most ${pn \choose k}$ independent $k$-sets $I \in \cI_k(G)$ with a subset $W \subset I$ such that $\gap W$ is an $A$-transversal and $o_B(W) \ge ck^2$.
\end{lemma}

\begin{proof}
By Lemma~\ref{lem:manyopen:probbound}, for any choice of the graph $B$ and set $W \subset V(G)$ such that $W$ is an $A$-transversal of size at most $k$ and $o_B(W) \ge ck^2$, we have 
$$\Pr_A\big( \gap W \in \cI(G) \text{ and }\gaaap \Delta(A) \le 2pm \gap \big) \le e^{-c^2pk^2}.$$
Recall from Lemma~\ref{lem:D:holds} that $\Delta(A) \le 2pm$ holds with high probability, and observe that if this bound holds, then the expected number of independent $k$-sets in $G$ that contain a set $W$ as in the lemma is at most
$$3^k {n \choose k} e^{-c^2pk^2} \le \bigg( \frac{9n}{k} \cdot e^{-c^2pk} \bigg)^k \ll {pn \choose k}$$
since $pk = C \log k$ and $C > 2/c^2$. The lemma therefore follows by Markov's inequality. 
\end{proof}

Finally, we need to deal with case~$(e)$, which is a little trickier. 

\begin{lemma}\label{lem:counting:e}
With high probability there are at most $2^{O(k)} {pn \choose k}$ independent $k$-sets in $G$ that contain disjoint sets $S$ and $T$ with
\begin{equation}\label{eq:counting:e:sizes}
|S| \ge pm \qquad \text{and} \qquad |T| \ge ck,
\end{equation}
such that $S \cup T$ is an $A$-transversal, and moreover
$$o_B(S \cup T) \ge \frac{|S||T|}{32} \qquad \text{and} \qquad \tau\pig( I \setminus N_B(v) \pig) \le 3|S|$$
for some $v \in V(G)$.
\end{lemma}

\begin{proof}
To choose the elements of $I$, we first reveal the graph $B$, choose a vertex $v \in V(B)$, and then choose $k - t$ elements of $N_B(v)$ for some $t \le k$. Next we choose the $3|S|$ fibres that cover the set $I \setminus N_B(v)$, and $t$ elements in the union of these fibres. This gives at most
$$m \sum_{t \eqa 0}^k {2pn \choose k - t} \sum_{s \eqa pm}^k {2m \choose 3s} {3sn/m \choose t}$$
choices for the set $I$, if $\Delta(B) \le 2pm$. Since $pk = C \log k$ and $m = \Theta\pig( p^{-2} \pig)$, this is at most
$$2^{O(k)} {pn \choose k} \sum_{s \eqa pm}^k \sum_{t \eqa 0}^k k^{6s} \bigg( \frac{k}{pn} \cdot \frac{p^2sn}{t} \bigg)^t.$$
Now, by Lemma~\ref{lem:manyopen:probbound}, for any choice of the graph $B$ and disjoint sets $S,T \subset V(G)$ such that $S \cup T$ is an $A$-transversal, satisfying~\eqref{eq:counting:e:sizes} and with $o_B(S \cup T) \ge |S||T|/32$, we have 
\begin{equation}\label{eq:counting:e}
\Pr_A\big( \gap S \cup T \in \cI(G) \text{ and }\gaaap \Delta(A) \le 2pm \gap \big) \le e^{-c^2p|S||T|},
\end{equation}
since $|S||T| = \Omega\pig( k^2 / \log k \pig) \gg k^{3/2}$. 

Taking a union bound over the choice of $S$ and $T$, and recalling that $|T| \ge ck$, it follows that (conditional on the event $\cD$ holding) the expected number of independent $k$-sets as in the lemma is at most 
$$2^{O(k)} {pn \choose k} \sum_{s \eqa pm}^k \sum_{t \eqa 0}^k k^{6s} \bigg( \frac{psk}{t} \bigg)^t e^{-c^3psk}.$$
We now split the sum over $s$ into two parts. For $s \le C/p$, we have
$$\sum_{s \eqa pm}^{C/p} \sum_{t \eqa 0}^k k^{6s} \bigg( \frac{psk}{t} \bigg)^t e^{-c^3psk} \le \sum_{t \eqa 0}^k \bigg( \frac{Ck}{t} \bigg)^t = 2^{O(k)},$$
since $pk = C \log k$ and $C \ge 4 / c^3$. On the other hand, for each $s \ge C/p$ we have 
$$\sum_{t \eqa 0}^k \bigg( \frac{psk}{t} \bigg)^t \le k \cdot \pig( ps \pig)^k \le e^{c^4psk},$$
since $C \ge 1 / c^5$, and hence 
$$\sum_{s \eqa C/p}^k \sum_{t \eqa 0}^k k^{6s} \bigg( \frac{psk}{t} \bigg)^t e^{-c^3psk} \le \sum_{s \eqa C/p}^k k^{6s} e^{-c^3psk/2} = o(1),$$
since $pk = C \log k$ and $C \ge 8 / c^3$. By~\eqref{eq:counting:e} and Markov, this proves the lemma.  
\end{proof}

We can now put the pieces together and complete the proof of Theorem~\ref{thm:acmgraph}.

\begin{proof}[Proof of Theorem~\ref{thm:acmgraph}]
By Lemmas~\ref{lem:bipartite:lemma} and~\ref{lem:D:holds}, the event $\cA$ holds with high probability, and by Lemma~\ref{lem:structural:lemma}, if $\cA$ holds then every independent $k$-set of $G$ satisfies one of the five properties $(a)$--$(e)$. Finally, by Lemmas~\ref{lem:counting:a}--\ref{lem:counting:e}, it follows that with high probability $G$ has at most $2^{O(k)} {pn \choose k}$ independent sets of size $k$, as required.  
\end{proof} 

\subsection*{Statement on AI use} 

The construction and proof strategy were developed by the authors in the traditional way, without any use of AI tools. While writing down the details, ChatGPT-5.6 Sol was used to simplify the proofs of some of the lemmas. In particular, Sol suggested the simple and elegant variant of the method of~\cite[Lemma~3.4]{MSV} that we used in the proof of Lemma~\ref{lem:good_to_probability}, and also the idea of proving Lemma~\ref{lem:bipartite:lemma} via Lemma~\ref{lem:bipartite:A}. The final version of the paper was written entirely by the authors.


\begin{thebibliography}{99}

\bibitem{AKSz} M.~Ajtai, J.~Koml{\'o}s and E.~Szemer{\'e}di, A note on Ramsey numbers, \emph{J.~Combin.~Theory Ser.~A}, \textbf{29} (1980), 354--360.


\bibitem{AR} N.~Alon and V.~R\"odl, Sharp bounds for some multicolour Ramsey numbers, \emph{Combinatorica}, \textbf{25} (2005), 125--141.





\bibitem{Br26} D.~Brada\v{c}, Off-diagonal Ramsey numbers, arXiv:2605.28793.


\bibitem{CJMS25} M.~Campos, M.~Jenssen, M.~Michelen and J.~Sahasrabudhe, A new lower bound for the Ramsey numbers $R(3,k)$, arXiv:2505.13371.

\bibitem{CJMPS} M.~Campos, M.~Jenssen, M.~Michelen, F.~Pfender and J.~Sahasrabudhe, A polynomial improvement for the odd cycle-complete Ramsey numbers, \emph{Combinatorica}, to appear.




\bibitem{E59} P.~Erd\H{o}s, Graph theory and probability, \emph{Canad. J. Math.}, \textbf{11} (1959), 34--38.

\bibitem{E61} P.~Erd\H{o}s, Graph theory and probability II, \emph{Canad. J. Math.}, \textbf{13} (1961), 346--352.

\bibitem{ESz} P.~Erd\H{o}s and G.~Szekeres, A combinatorial problem in geometry, \emph{Compositio Math.}, \textbf{2} (1935), 463--470.



\bibitem{HW} X.~He and Y.~Wigderson, Multicolor Ramsey Numbers via Pseudorandom Graphs, \emph{Electronic J.~Combin.}, \bf 27 \rm (2020), P1.32.

\bibitem{HHKP} Z.~Hefty, P.~Horn, D.~King and F.~Pfender, Improving $R(3,k)$ in just two bites, arXiv:2510.19718.

\bibitem{Kim} J.H.~Kim, The Ramsey number $R(3,t)$ has order of magnitude $t^2/\log t$, \emph{Random Structures Algorithms}, \textbf{7} (1995), 173--207.

\bibitem{KSSW} M. K\"uhn, L. Sauermann, R. Steiner and Y. Wigderson, Disproof of the Odd Hadwiger Conjecture, arXiv:2512.20392


\bibitem{ICM26} R.~Morris, Some recent results in Ramsey theory, \emph{Proceedings of the International Congress of Mathematicians}, Philadelphia, 2026, arXiv:2601.05221.

\bibitem{MSV} R. Morris, J. Sahasrabudhe and J. Verstra\"ete, On the Erdős--Rogers function, arXiv:2607.16118.

\bibitem{R30} F.P.~Ramsey, On a problem of formal logic, \emph{Proc. London Math. Soc.}, \textbf{30} (1930), 264--286.

\bibitem{Sh83} J.B.~Shearer, A note on the independence number of triangle-free graphs, \emph{Discrete Math.}, \textbf{46} (1983), 83--87.


\bibitem{SpR3k} J.~Spencer, Eighty Years of Ramsey $R(3,k)...$ and Counting!, In: Ramsey Theory: Yesterday, Today, and Tomorrow, pp. 27--39. Boston, MA: Birkh\"auser Boston, 2011.

\end{thebibliography}
\end{document}